\documentclass[a4paper,11pt]{amsart}
\usepackage{mathrsfs}
\usepackage[all]{xy}
\usepackage{amsmath,amssymb,amscd,bbm,amsthm,mathrsfs}
\usepackage{graphicx}
\newtheorem{thm}{Theorem}[section]
\newtheorem{lem}{Lemma}[section]
\newtheorem{cor}{Corollary}[section]
\newtheorem{prop}{Proposition}[section]
\newtheorem{rem}{Remark}[section]

\begin{document}
\numberwithin{equation}{section}

\title[ The Linear algebra in  the quaternionic  pluripotential theory]
 {The  Linear algebra in  the quaternionic  pluripotential theory}
\author{  Wei Wang}
\thanks{
Supported by National Nature Science Foundation in China (No.
11571305)  }\thanks{   Department of Mathematics,
Zhejiang University, Zhejiang 310027,
 P. R. China, Email:   wwang@zju.edu.cn}

\begin{abstract} We clarify  the  linear algebra used in  the quaternionic  pluripotential theory so that proofs of several results there can be greatly simplified. In particular, we characterize and normalize real $2$-forms with respect to the quaternionic structure,  and show  that  the Moore determinant of   a quaternionic hyperhermitian matrix  is the coefficient  of  the exterior product of the associated real $2$-form. As a corollary, the quaternionic Monge-Amp\`{e}re operator is the coefficient  of the exterior product of the Baston operator.
\end{abstract}
\maketitle
\begin{center}
\begin{minipage}{135mm}
\end{minipage}
\end{center}
\section{Introduction}

Alesker \cite{alesker1} \cite{alesker4} proved   a quaternionic version of Chern-Levine-Nirenberg estimate and extended the definition of the quaternionic Monge-Amp\`{e}re operator to continuous quaternionic plurisubharmonic functions. The quaternionic Monge-Amp\`{e}re operator on quaternionic manifolds was introduced and investigated by Alesker, Shelukhin and Verbitsky et. al. \cite{AS}-\cite{alesker6}. More generally, it is interesting to develope pluripotential theory  on calibrated manifolds \cite{harvey1}-\cite{harvey2}. By using the twistor method  in complex geometry, Alesker observed \cite{alesker2} (see also \cite{AS2}) the coincidence between the quaternionic Monge-Amp\`{e}re operator  and the  $n$-th exterior power of  the Baston operator $\triangle $.
The Baston operator $\triangle$ is the first operator of $0$-Cauchy-Fueter complex, which is known  explicitly \cite{baston} \cite{bS} \cite{Wang} \cite{wang-mfd}.
Motivated by this fact, we \cite{wan-wang} introduced the first-order differential operators $d_0$ and $d_1$ acting on the quaternionic version of differential forms. The behavior of $d_0,d_1$ and $\triangle=d_0d_1$ is very similar to $\partial,\overline{\partial}$ and $\partial \overline{\partial}$ in several complex variables,  and several results in the complex pluripotential theory (cf. e.g. \cite{bed} \cite{klimek}) were extended to the quaternionic case \cite{wan4}-\cite{wan3}. The purpose of this paper is to clarify the  linear algebra used in  the quaternionic  pluripotential theory in \cite{wan-wang} to simplify proofs of several results there.

Denote by $M_{\mathbb{F}}(p,m)$ the space of $\mathbb{F}$-valued $(p\times m)$-matrices, where $\mathbb{F}=\mathbb{R},\mathbb{C},\mathbb{H}$.
Motivated by
 the well known embedding of the quaternionic algebra $\mathbb{H}$ into End$(\mathbb{C}^{2})$ given by
\begin{equation}\label{2.1}x_0+\textbf{i}x_1+\textbf{j}x_2+\textbf{k}x_3 \mapsto\left(
                                                                                 \begin{array}{rr}
                                                                                   x_0+\textbf{i}x_1 & -x_2-\textbf{i}x_3 \\
                                                                                   x_2-\textbf{i}x_3 & x_0-\textbf{i}x_1 \\
                                                                                 \end{array}
                                                                               \right),
\end{equation}we define the embedding
\begin{equation*}
   \tau:M_{\mathbb{F}}(p,m)\longrightarrow M_{\mathbb{C}}(2p,2m)
\end{equation*}
as follows.
For a quaternionic $p\times m$-matrix $\mathcal{M}$, we   write $\mathcal{M}=a+b\mathbf{j}$ for some    complex matrices $a, b\in M_{\mathbb{C}}( p, m)$.
Then   $ {\tau}(\mathcal{M})$ is  the complex $(2p\times 2m)$-matrix
\begin{equation}\label{eq:tau-A}
    {\tau}(\mathcal{M}):= \left(
                                       \begin{array}{rr}
                                          {a} &-{b} \\
                                         \overline {b} & \overline{{a}} \\
                                       \end{array} \right).
\end{equation}

Let $\wedge^{ 2k}\mathbb{C}^{2n}$ be the $ k$-th exterior product of $\mathbb{C}^{2n}$, $ k=1,\ldots,  n$.
There exists a real linear action $ \rho(\textbf{j})$ on $\mathbb{C}^{2n}$ naturally (\ref{eq:rouj}), which induces an action on $\wedge^{2k}\mathbb{C}^{2n}$.
An element $\varphi$ of $\wedge^{2k}\mathbb{C}^{2n}$ is called  {\it real} if $\rho(\textbf{j})\varphi=\varphi$. Denote by $\wedge^{2k}_{\mathbb{R}}\mathbb{C}^{2n}$ the subspace of all real elements in $\wedge^{2k}\mathbb{C}^{2n}$, which  is the counterpart of the space of $(k,k)$-forms in complex analysis.
Fix a basis
$\{\omega^0,\omega^1,\ldots,\omega^{2n-1}\}$ of $\mathbb{C}^{2n}$.
We characterize and normalize real $2$-forms as follows.
\begin{thm} \label{thm:real-form} (1) For a complex skew symmetric matrix  $M =(M_{ AB})\in M_{\mathbb{C}}(2n,2n )$, the $2$-form
\begin{equation}\label{eq:real-2form}
   \omega =\sum_{A,B=0}^{2n-1} M_{ AB}\,\omega^A\wedge\omega^B
\end{equation}  is real if and only if there exists a hyperhermitian $n\times n$-matrix $\mathcal{M}=({\mathcal{M}}_{jk})$, i.e. ${\mathcal{M}}_{jk}=\overline{\mathcal{M}_{kj}}$, such that
\begin{equation}\label{eq:real-2form'}
   M=\tau(\mathcal{M})J,
\end{equation}where
\begin{equation}\label{eq:J-matrix}
   J=\left(
                                       \begin{array}{cc}
                                           0& I_n \\
                                      -I_n& 0\\
                                       \end{array} \right).
\end{equation}

 (2) When $\omega$  in (\ref{eq:real-2form}) is  real, there exists a quaternionic unitary matrix $\mathcal{E}\in \text{U}_{\mathbb{H}}(n)$ such that
\begin{equation}\label{eq:normal-matrix}
   \tau(\mathcal{E})^tM\tau(\mathcal{E})=\left(
    \begin{array}{cc}
                                         0 &
    \begin{array}{ccc}\nu_0& &\\&\ddots&
                                         \\&  & \nu_{n-1} \\
                                                                               \end{array}  \\
                                            \begin{array}{ccc}-\nu_0& &\\&\ddots&
                                         \\&  & -\nu_{n-1} \\
                                                                               \end{array}   &0 \\
                                       \end{array} \right)
\end{equation}for some real numbers $\nu_0,\ldots,\nu_{n-1}$.
Namely, we can normalize $\omega$ as
\begin{equation}\label{eq:normal-form}
                                         \omega =2\sum_{ l=0}^{n-1} \nu_{ l}\widetilde{\omega}^l\wedge\widetilde{\omega}^{l+n}
                                       \end{equation}
    with $\widetilde {\omega}^A= \mathcal{E^*}.{\omega}^A$. In particular, $\omega$ is strongly positive if and only if $\nu_{ l}\geq 0$, $l=0£¬\ldots,n-1$, and at least one is positive.
 \end{thm}

 The Moore determinant denoted by det is a determinant defined for quaternionic  hyperhermitian matrices
and takes real values.
Because of the noncommutativity of quaternionic numbers, the Moore determinant   is very complicated. Here we find a concrete expression of the Moore determinant of  a quaternionic  hyperhermitian matrix as follows.

More generally, consider the homogeneous polynomial $\text{det}(\lambda_1\mathcal{M}_1+\ldots+\lambda_n\mathcal{M}_n)$ in real variables $\lambda_1,\ldots,\lambda_n$ of degree $n$. The coefficient of the monomial $\lambda_1\cdots\lambda_n$ divided by $n!$ is called the {\it mixed discriminant} of the hyperhermitian matrices $\mathcal{M}_1,\ldots,\mathcal{M}_n$, and it is denoted by $\text{det}(\mathcal{M}_1,\ldots,\mathcal{M}_n)$. In particular, when $\mathcal{M}_1=\ldots=\mathcal{M}_n=\mathcal{M }$, $\text{det}(\mathcal{M}_1,\ldots,\mathcal{M}_n)=\text{det}(\mathcal{M })$.
For  skew symmetric matrices $M^{(t)} \in M_{\mathbb{C}}(2n,2n)$, $t=1,\ldots, n$, such that  $2$-forms $\omega_t=\sum_{A,B=0}^{2n-1} M^{(t)}_{AB}\, \omega^A\wedge\omega^B  $  are real, we define $\bigtriangleup_n(M^{(1)} ,\ldots,M^{(n)} )$ by
\begin{equation}\label{eq:triangleup-n}\bigtriangleup_n\left(M^{(1)} ,\ldots,M^{(n)} \right)\Omega_{2n}:=
   \omega_1\wedge\cdots\wedge\omega_n,
\end{equation}where\begin{equation}\label{eq:omega}\Omega_{2n}:=\omega^0\wedge\omega^{  n }\cdots\wedge
\omega^{ n-1} \wedge
\omega^{2n-1}\in \wedge^{2n}_{\mathbb{R}}\mathbb{C}^{2n}.\end{equation}

 \begin{thm} \label{thm:det}For    hyperhermitian matrices $\mathcal{M}_1,\ldots, \mathcal{M}_n\in M_{\mathbb{ H}}( n, n )$,   we have
    \begin{equation}\label{eq:det}2^n n!~\text{det}~( \mathcal{M}_1,\ldots,\mathcal{M}_n)=
       \triangle_n\left( \tau(\mathcal{M}_1)J,\ldots, \tau(\mathcal{M}_n)J  \right).
    \end{equation}
 \end{thm}

Alesker introduced in \cite{alesker1} the mixed quaternionic
Monge-Amp\`{e}re operator det$(u_1,\ldots,u_n)$ for  real $C^2$ functions
$u_1,\ldots,u_n $ as
\begin{equation*}\text{det}(u_1,\ldots,u_n):=\text{det}~\left(\left(\frac{\partial^2u_1}{\partial \overline{{ q_l}}\partial {q}_k}(q)\right),\ldots,\left(\frac{\partial^2u_n}{\partial  \overline{ {q_l}}\partial {q}_k}(q)\right)\right).\end{equation*}
 Recall that  for a real $ C^2$ function $u$, the {\it Baston operator} is
  \begin{equation}\label{eq:triangle u}\triangle u= \sum_{A,B=0}^{2n-1}   \triangle_{AB} u\, \omega^A\wedge\omega^B  ,\end{equation}
 where
 \begin{equation}\label{2.10}\triangle_{AB}u:=\frac{1}{2}\left(\nabla_{A 0' }\nabla_{ B 1' }u-\nabla_{B0' }\nabla_{  A1' }u\right),
\end{equation}i.e.
$2\triangle_{AB}$ is the determinant of the $(2\times2)$-submatrix of $A$-th and $B$-th rows in
\begin{equation}\label{eq:nabla-jj'-new}\left(
                             \begin{array}{cc}
                               \nabla_{00' } & \nabla_{01' } \\
                                                              \vdots&\vdots\\
                               \nabla_{ l 0' } & \nabla_{ l 1' } \\
                               \vdots&\vdots\\ \nabla_{ n 0' } & \nabla_{  n 1' } \\ \vdots&\vdots\\
                               \nabla_{(  n+l)0 '} & \nabla_{(  n+l)1 '} \\ \vdots&\vdots\\
                                                            \end{array}
                           \right):=\left(
                                      \begin{array}{cc}
                                      \partial_{x_{0}}+\textbf{i}\partial_{x_{1}} & -\partial_{x_{2}}-\textbf{i}\partial_{x_{3}} \\
                                                                                \vdots&\vdots\\
                                         \partial_{x_{4l}}+\textbf{i}\partial_{x_{4l+1}} & -\partial_{x_{4l+2}}-\textbf{i}\partial_{x_{4l+3}} \\ \vdots&\vdots    \\ \partial_{x_{2}}-\textbf{i}\partial_{x_{3}} & \partial_{x_{0}}-\textbf{i}\partial_{x_{1}} \\   \vdots&\vdots\\
                                        \partial_{x_{4l+2}}-\textbf{i}\partial_{x_{4l+3}} & \partial_{x_{4l}}-\textbf{i}\partial_{x_{4l+1}} \\
                                        \vdots&\vdots
                                      \end{array}
                                    \right).
\end{equation}
$\nabla_{A\alpha }$'s above and
$\tau(\mathcal{M})$ in (\ref{eq:tau-A}) are relabeling of corresponding notations in \cite{wan-wang} (cf. Remark \ref{rem:relabel}). We use this relabeling in this paper because it makes some of proofs more transparent.

\begin{thm}\label{thm:det'} Let $u_1,u_2,\ldots,u_n$ be real $C^2$ functions on
$\mathbb{H}^n$. Then we have:
\begin{equation}\label{2.17}\triangle
u_1\wedge\ldots\wedge\triangle
u_n=n!~\text{det}~(u_1,u_2,\ldots,u_n)\Omega_{2n}.
\end{equation}
\end{thm}

The complexified version of this theorem was proved by
Alesker  \cite{alesker2}   by using the twistor method of complex geometry. The proof given in the appendix (Theorem A.1) of \cite{wan-wang} is elementary  except for using the fact that the linear combinations of delta
functions supported on quaternionic  hyperplanes are dense in the space of generalized functions. Here we give a completely elementary proof   only by using linear algebra. In Section 2, we prove   characterization and normalization of  real $2$-forms  in Theorem \ref{thm:real-form}, by which we can prove Theorem \ref{thm:det}.
Theorem \ref{thm:det'} is proved in Section 3. We also give a very simple proof of the fact that $-\frac{1}{\|q\|^2}$   is the fundamental solution of the quaternionic Monge-Amp\`{e}re operator. In Section 4, we give us a simplified proof of the invariance of $d_\alpha$ and $\triangle$ under quaternionic linear transformations.

\section{Real   forms and  the Moore determinant }
Denote by $\text{GL}_\mathbb{H}(n)$ the set of all invertible quaternionic $(n\times
n)$-matrices, and denote by $\text{U}_{\mathbb{H}}(n)$ the set of all unitary quaternionic $(n\times
n)$-matrices, i.e.,
  \begin{equation*}
     \text{U}_{\mathbb{H}}(n)=\{\mathcal{M}\in\text{GL}_\mathbb{H}(n), {\mathcal{M}}^*\mathcal{M}=\mathcal{M} {\mathcal{M}}^*=I_{n }\},
\end{equation*}
  where $ {\mathcal{M}}^*=\overline{\mathcal{M}}^t$, i.e. $( {\mathcal{M}}^*)_{jk}=\overline{\mathcal{M}_{kj}}$.  Here $\overline{q}=x_0-x_{ 1}\mathbf{i}-x_{ 2}\mathbf{j}-x_{ 3 }\mathbf{k}$ is the {\it quaternionic  conjugate} of $q=x_0+x_{ 1}\mathbf{i}+x_{ 2}\mathbf{j}+x_{ 3 }\mathbf{k}$ .

\begin{prop} \label{prop:diagonal} {\rm (Claim 1.1.4, 1.1.7  in \cite{alesker1})} For a hyperhermitian $(n\times n)$-matrix $\mathcal{M}$, there exists a unitary matrix $\mathcal{E}$ such that $\mathcal{E}^*\mathcal{M}\mathcal{E}$ is diagonal and real.
\end{prop}
\begin{prop} \label{prop:Moore} {\rm (Theorem 1.1.9 in \cite{alesker1})} (1) The Moore determinant of  any complex hermitian  matrix   considered as a quaternionic hyperhermitian matrix  is equal to its usual determinant.

(2) For any quaternionic hyperhermitian $(n\times n)$-matrix $\mathcal{M}$ and any quaternionic $(n\times n)$-matrix $\mathcal{C}$
\begin{equation*}
     \det(\mathcal{C}^*\mathcal{M}\mathcal{C})=\det(\mathcal{ A} )\det(\mathcal{C}^* \mathcal{C}).
\end{equation*}
\end{prop}

\begin{prop}\label{prop:tau} $(1)$  For a quaternionic $(p\times m)$-matrix $\mathcal{M}$ and a quaternionic $(m\times l)$-matrix $\mathcal{N}$, we have
 \begin{equation}\label{eq:rep}
    \tau(\mathcal{MN})=\tau(\mathcal{M})\tau(\mathcal{N}).
 \end{equation}

$(2)$ For $M\in M_\mathbb{C}(2n,2n)$, there exists  $\mathcal{M}\in M_\mathbb{H}(n,n)$ such that $M=\tau(\mathcal{M})$ if and only if \begin{equation}\label{eq:AJ}J\overline{  M }= M J.\end{equation}
$(3)$  $\tau\left( {\mathcal{M}}^*\right)=\overline{\tau(\mathcal{M})}^t$ for any quaternionic $(n\times n)$-matrix $\mathcal{M}$.   $\mathcal{M}\in\text{U}_{\mathbb{H}}(n)$ if and only if  $\tau(\mathcal{M})$ is  complex unitay and symplectic, i.e.
\begin{equation}\label{eq:symplectic}
   \tau(\mathcal{M})J\tau(\mathcal{M})^t=J.
\end{equation}
\end{prop}
\begin{proof} Write $\mathcal{M}=a+b\mathbf{j}$ as in (\ref{eq:tau-A}).
If we write the quaternionic $(m\times l)$-matrix  $\mathcal{N}=c+d\mathbf{j}$ for some    complex matrices $c, d\in M_{\mathbb{C}}( m,l)$, we have
\begin{equation}\label{eq:AB}\mathcal{MN}=
 (a+b\mathbf{j})(  c+d\mathbf{j})= ac-b\overline{d}+(ad+b\overline{c})\mathbf{j},
\end{equation}by
\begin{equation}\label{eq:dj-jd}
 \mathbf{j}  d=\overline{d}\mathbf{j}.
\end{equation}
Then (\ref{eq:AB})
implies
\begin{equation*}
   {\tau}(\mathcal{MN})= \left(
                                       \begin{array}{rr}ac-b\overline{d}
  &-  ad-b\overline{c} \\
                                       \overline  {ad}+\overline b{c} & \overline{ac}-\overline b{d} \\
                                       \end{array} \right)= \left(
                                       \begin{array}{rr}
{a} &-{ b} \\
                                       \overline   {b} & \overline{a} \\
                                       \end{array} \right)  \left(
                                       \begin{array}{rr}
{ c} &-{ d} \\
                                     \overline   {d} & \overline{c} \\
                                       \end{array} \right)= {\tau}(\mathcal{M}) {\tau}(\mathcal{N}),
\end{equation*} by definition of $\tau$.

(2) If we write
\begin{equation*}
  M=\left(
                                       \begin{array}{rr}
                                         a & b \\
                                        c&d\\
                                       \end{array} \right),
\end{equation*} where $a,b,c,d$ are $n\times n$ complex matrices, then
(\ref{eq:AJ}) is equivalent to $c=-\overline{b}, d=\overline{a}$, i.e.
\begin{equation*}
  M=\left(
                                       \begin{array}{rr}
                                         a & b \\
                                       -\overline{b}&\overline{a}\\
                                       \end{array} \right).
\end{equation*}Thus $M=\tau(\mathcal{M})$ with $\mathcal{M}=a-b\mathbf{j}$. The converse is also true.

(3) By definition, the quaternionic  conjugate of $ {\mathcal{M}} =a+b\mathbf{j}$  is $\overline{\mathcal{M}} = \overline{ a}-{b} \mathbf{j}.$ Then we have
\begin{equation}\label{eq:A-bar}
\overline{\mathcal{M}}^t= \overline{ a}^t-{b}^t\mathbf{j},
\end{equation}
 and so
\begin{equation}\label{eq:A-bar'}
   {\tau}(\mathcal{M }^*)= \left(
                                       \begin{array}{rr}
\overline{{a}}^t & { b}^t \\
                                     -  \overline   {b}^t &  {a}^t \\
                                       \end{array} \right)=\overline{{\tau}(\mathcal{M })}^t .
\end{equation}
${\tau}(\mathcal{M })$  is unitary by $\overline{{\tau}(\mathcal{M }) }^t{\tau}(\mathcal{M })= {\tau}(\mathcal{M }^*\mathcal{M })=I_{2n}$.
The symplecticity (\ref{eq:symplectic}) of ${\tau}(\mathcal{M })$ follows from (\ref{eq:A-bar'}) and (\ref{eq:AJ}).
The converse is also true.
\end{proof}

\begin{rem}\label{rem:relabel}
This Proposition is essentially a conjugate and relabelling version of proposition 2.1 in \cite{wan-wang}, where   $J$ equals to
\begin{equation*}  \left(
                                       \begin{array}{ccccccc}
                                       0 & 1 &   &   &   \\
                                         -1 & 0 &   &   &   \\
                                           &   & 0 & 1 &   \\
                                           &   & -1 & 0 &   \\
                                           &   &   &   & \ddots \\
                                           & & &&&0&1\\
                                           & & &&&-1&0\\\end{array}
                                     \right),
\end{equation*}
and the embedding $\tau$ maps  a quaternionic $(l\times m)$-matrix $\mathcal{A}=(\mathcal{A}_{jk})_{l\times m}$   to a complex $(2l\times 2m)$-matrices \begin{equation*} \tau(\mathcal{A})=\left(
                                                              \begin{array}{ccc}
                                                                \tau(\mathcal{A}_{00}) & \tau(\mathcal{A}_{01}) & \cdots   \\
                                                                \tau(\mathcal{A}_{10}) & \tau(\mathcal{A}_{11}) & \cdots  \\
                                                                \cdots & \cdots & \ddots   \\
                                                              \end{array}
                                                            \right).
\end{equation*}
Here $\tau(\mathcal{A}_{jk})$ is the complex $(2\times2)$-matrix \begin{equation*} \left(
                                       \begin{array}{rr}
                                         a_{jk}^0-\textbf{i}a_{jk}^1 & -a_{jk}^2+\textbf{i}a_{jk}^3 \\
                                         a_{jk}^2+\textbf{i}a_{jk}^3 & a_{jk}^0+\textbf{i}a_{jk}^1 \\
                                       \end{array}
                                     \right) \end{equation*}
if we write $
    \mathcal{A}_{jk}=a_{jk}^0+\textbf{i}a_{jk}^1+\textbf{j}a_{jk}^2+\textbf{k}a_{jk}^3\in\mathbb{H}.
 $ One advantage of using $\tau$ as in (\ref{eq:tau-A}) is that the proof of Proposition \ref{prop:tau} becomes more transparent.
 \end{rem}

For $ M\in M_{\mathbb{C}}(2n,2n )$, define its {\it $\mathbb{C}$-linear  action} on $\mathbb{C}^{2n}$ as:
\begin{equation}\label{eq:A.}M.\omega^A=\sum_{B=0}^{2n-1}M_{AB}\, \omega^B,\end{equation}and define the  induced action on $\wedge^{2k}\mathbb{C}^{2n}$ as
 \begin{equation*}
 {M}.\left(\omega^{A_1}\wedge \ldots\wedge\omega^{A_{2k}}\right)= M.\omega^{A_1}\wedge \ldots\wedge{M}.\omega^{A_{2k}} .
 \end{equation*}
 For $\mathcal{M}\in M_{\mathbb{H}}( n, n )$, define its \emph{induced $\mathbb{C}$-linear  action}  on $\mathbb{C}^{2n}$ as
\begin{equation}\label{eq:A.}\mathcal{M}.\omega^A= \tau(\mathcal{M}). \omega^A,\end{equation}and  so on $\wedge^{2k}\mathbb{C}^{2n}$.

  Note that for $\mathcal{M}\in\text{U}_{\mathbb{H}}(n)$,
$\mathcal{M}.\beta_n  =\beta_n
$    for $\beta_n$  given by
\begin{equation}\label{eq:beta-d}
    \beta_n:=\sum_{l=0}^{n-1}  \omega^l\wedge\omega^{n+l},
\end{equation}
since $\tau(\mathcal{M})$
is symplectic by Proposition \ref{prop:tau} (3) and so
\begin{equation}\label{eq:beta-inv}\begin{split}\mathcal{M}.\beta_n &=\sum_{ l=0}^{n-1} \sum_{A,B} \tau( \mathcal{M})_{lA} \tau( \mathcal{M} )_{(n+l)B}\,{\omega}^A\wedge {\omega}^{B}\\&=\frac 12\sum_{A,B} \left[\tau( \mathcal{M} )^tJ \tau( \mathcal{M} )\right]_{AB}\,{\omega}^A\wedge {\omega}^{B}=
   \sum_{ l=0}^{n-1}  {\omega}^l\wedge {\omega}^{l+n}.
\end{split}\end{equation}
  Consequently $\mathcal{M}.(\wedge^n
\beta_n)=\wedge^n \beta_n$, i.e.,
\begin{equation}\label{eq:A-omega}\mathcal{M}.\Omega_{2n}=\Omega_{2n},\end{equation}where $\Omega_{2n}$ is given by (\ref{eq:omega}). This means that $\beta_n $ and $\Omega_{2n}$ are invariant under
unitary transformations on $\mathbb{H}^n$.

Recall that $\mathbf{j}$
  induces a real linear map \cite{wan-wang}
\begin{equation}\label{eq:rouj}\rho(\textbf{j}):\mathbb{C}^{2n}\rightarrow\mathbb{C}^{2n},
\qquad\rho(\textbf{j})(z\omega^k)=\overline{z}J.\omega^k.
\end{equation}
$\mathbf{j}$ is not $\mathbb{C}$-linear, but it is anti-complex linear. It induces an   action on $\wedge^{ k}\mathbb{C}^{2n}$.
$\wedge^{2k}_{\mathbb{R}}\mathbb{C}^{2n}$ is the subspace of all real elements in $\wedge^{2k}\mathbb{C}^{2n}$. $\beta_n$ above is real,  and so is $\Omega_{2n}$.

 A quaternionic  $(m\times k)$-matrix $\mathcal{M}\in M_{\mathbb{H}}( m,k)$
defines a
map $\eta:\mathbb{H}^{k}\rightarrow \mathbb{H}^{m}$  given by
\begin{equation*}
   \eta(q_1,\dots q_k)=\left(\cdots, \sum_{l=1}^{k} \mathcal{M}_{jl}q_l,\cdots\right),
\end{equation*} which  is right $\mathbb{H}$-linear, i.e. $\eta(q_1q,\dots q_kq)=\eta(q_1,\dots q_k)q$ for any $q\in \mathbb{H}$.
It induces a $\mathbb{C}$-linear map $\tau(\mathcal{M}):\mathbb{C}^{2k}\rightarrow\mathbb{C}^{2m}$.     Similar to (\ref{eq:A.}), the induced $\mathbb{C}$-linear pulling back transformation   $\eta^*:\mathbb{C}^{2m}\rightarrow\mathbb{C}^{2k}$ is defined as:
\begin{equation}\label{g^*}
\eta^*\widetilde{\omega}^p=\sum_{j=0}^{2k-1}\tau(\mathcal{M})_{pj}\omega^j,\qquad p=0,\ldots,2m-1,
\end{equation}
where $\{\widetilde{\omega}^0,\ldots,\widetilde{\omega}^{2m-1}\}$ is a basis of $\mathbb{C}^{2m}$ and $\{\omega^0,\ldots,\omega^{2k-1}\}$ is a basis of $\mathbb{C}^{2k}$. It  induces a $\mathbb{C}$-linear pulling back transformation  on $\wedge^{2k}\mathbb{C}^{2m}$   given by
$
 \eta^*(\alpha\wedge\beta)=\eta^*\alpha\wedge \eta^*\beta
$ inductively.

An element $\omega\in\wedge_{\mathbb{R}}^{2k}\mathbb{C}^{2n}$ is said to be \emph{elementary strongly positive} if there exist linearly independent right $\mathbb{H}$-linear mappings $\eta_j:\mathbb{H}^n\rightarrow \mathbb{H}$ , $j=1,\ldots,k$, so that \begin{equation}\label{eq:omega'}\omega=\eta_1^*\widetilde{\omega}^0\wedge \eta_1^*\widetilde{\omega}^1\wedge\ldots\wedge\eta_k^*\widetilde{\omega}^0\wedge \eta_k^*\widetilde{\omega}^1,\end{equation}where $\{\widetilde{\omega}^0,\widetilde{\omega}^1\}$ is a basis of $\mathbb{C}^{2}$. The following corollary easily follows from the definition.

\begin{cor}
   $\omega\in\wedge_{\mathbb{R}}^{2k}\mathbb{C}^{2n}$ is  elementary strongly positive  if and only if
\begin{equation}\label{eq:elementary-strongly-positive}
  \omega=\mathcal{M}.(\omega^0\wedge\omega^n\wedge\ldots\wedge\omega^{ k-1}\wedge\omega^{ n+k-1})
\end{equation}for some quaternionic matrix $\mathcal{M}\in M_{\mathbb{H}}(n)$ with first $k$ rows linearly independent.
\end{cor}

An element $\omega\in\wedge_{\mathbb{R}}^{2k}\mathbb{C}^{2n}$ is called \emph{strongly positive} if it belongs to the convex cone $\text{SP}^{2k}\mathbb{C}^{2n}$ in $\wedge_{\mathbb{R}}^{2k}\mathbb{C}^{2n}$ generated by elementary strongly positive $2k$-elements. An $2k$-element $\omega$ is said to be \emph{positive} if for any elementary strongly positive element $\eta\in \text{SP}^{2n-2k}\mathbb{C}^{2n}$, $\omega\wedge\eta$ is strongly positive.
By definition, $ \beta_n$ is a strongly  positive
$2$-form, and
$ \beta_n^n=\wedge^n\beta_n=n!~\Omega_{2n}$ is a positive
$2n$-form.

\vskip 5mm
 {\it Proof of Theorem \ref{thm:real-form}}. (1) For $ \omega$ given by (\ref{eq:real-2form}), note that
\begin{equation*}
  \rho(\textbf{j})\omega= \sum_{A,B=0}^{2n-1}\overline{M_{AB}} J.\omega^A\wedge J.\omega^B    =\sum_{C,D=0}^{2n-1} (J^t \overline{M }J)_{CD} \, \omega^C\wedge  \omega^D
\end{equation*}by definition of $\rho(\textbf{j})$.
So $\omega$ is a real $2$-form if and only if  $J^t \overline{M }J=M$, i.e.
\begin{equation*}\label{eq:AJ'}
   J\overline{M }= M J ,
\end{equation*}by $J^2=-I_{2n}$ and $J^t=-J$.
Then by Proposition \ref{prop:tau}  (2), we can write\begin{equation*}
 M=\left(
                                       \begin{array}{rr}
                                         a & -b \\
\overline{b}&\overline{a}\\
                                       \end{array} \right),
\end{equation*}for some    complex matrices $a,b\in M_{\mathbb{C}}( n,n)$, and so
\begin{equation}\label{eq:A-tau}
  M=\left(
                                       \begin{array}{rr}
                                        - b & - a\\
                                        \overline{a}&-\overline{b}\\
                                       \end{array} \right)J=\tau(-b +a\mathbf{j})J,
\end{equation}
where $\mathcal{M}:=-b+a\mathbf{j}$ is hyperhermitian. This is   because
  skew-symmetry of $M$ is equivalent to
\begin{equation}\label{eq:quaternioncity}
   b^t=\overline{b},  \qquad a^t=-a,
\end{equation}
and so
\begin{equation*}
  \left ( \overline{-b+ a\mathbf{j}} \right)^t =\left(-\overline{b}- { a }\mathbf{j}\right )^t =-b+ a\mathbf{j},
\end{equation*}
where the bar in the left hand side is the quaternionic  conjugate.

(2) By applying Proposition \ref{prop:diagonal} (1) to the hyperhermitian matrix $\mathcal{M}  $,  we can find  a quaternionic unitary matrix $\mathcal{B}=e+f\mathbf{j}\in U_{\mathbb{H}}(n)$ so that
\begin{equation}\label{eq:-b+aj}
   \overline{\mathcal{B}}^t( -b +a\mathbf{j})\mathcal{B}={\rm diag} (\nu_0,\ldots,\nu_{n-1})
\end{equation} for some real numbers $\nu_0,\ldots,\nu_{n-1}$.
Multiply $-\mathbf{j}$ from right on both sides to get
\begin{equation}\label{eq:-b+aj'}
  ( \overline{e}^t-{f}^t\mathbf{j}) (a+ b \mathbf{j} )  (\overline e+\overline{f }\mathbf{j})= {\rm diag}  (-\nu_1\mathbf{j},\ldots,-\nu_n\mathbf{j}),
\end{equation}by
\begin{equation}\label{eq:switch-j}
   ( e+f \mathbf{j})\mathbf{j}=\mathbf{j}(\overline e+\overline{f }\mathbf{j})\qquad {\rm  and }\qquad  -(- b +a \mathbf{j} )\mathbf{j}= a+ b \mathbf{j}
\end{equation}
 by using (\ref{eq:dj-jd}).
Apply  the mapping $ {\tau}$  as a representation (\ref{eq:rep})  on both sides of (\ref{eq:-b+aj'}) to get
\begin{equation}\label{eq:matrix-b+aj} \left(
    \begin{array}{rr}
                                         \overline e &-\overline{ f }\\
                                        {f} & {e} \\
                                       \end{array} \right)^t\left(
    \begin{array}{rr}
                                         a & -b \\
\overline{b} &\overline{a} \\
                                       \end{array} \right)\left(
    \begin{array}{rr}
                                         \overline e & -\overline{ f }\\
                                       {f} & {e} \\
                                       \end{array} \right)=\left(
    \begin{array}{cc}
                                      {\Large   0} &
    \begin{array}{ccc}\nu_0& &\\&\ddots&
                                         \\&  & \nu_{n-1} \\
                                                                               \end{array}  \\
                                            \begin{array}{ccc}-\nu_0& &\\&\ddots&
                                         \\&  & -\nu_{n-1} \\
                                                                               \end{array}   &0 \\
                                       \end{array} \right).
\end{equation}

Now set $\mathcal{E}:=\overline e+\overline{f }\mathbf{j}$. The left hand side of (\ref{eq:matrix-b+aj}) is exactly $\tau(\mathcal{E})^tM\tau(\mathcal{E})$.
We see that $\mathcal{E}$ is unitary since
\begin{equation*}
  \overline e+\overline{f }\mathbf{j} = - \mathbf{j}(  e+ {f }\mathbf{j})\mathbf{j},
\end{equation*}
  and    $ e+ {f }\mathbf{j}$ and $\mathbf{j}$ are both unitary. So (\ref{eq:normal-matrix}) holds. At last
we find that
\begin{equation*}
   {\omega}^A =\tau(\mathcal{E}).\widetilde{\omega}^A
\end{equation*}
by $\tau(\mathcal{E}).\tau(\mathcal{E}^*).{\omega}^A=\tau(\mathcal{E} \mathcal{E}^*).{\omega}^A={\omega}^A$ and
  \begin{equation*}
  \begin{split}
    \omega& =\sum_{A,B=0}^{2n-1}M_{AB}   \tau(\mathcal{E}).\widetilde{\omega}^A  \wedge \tau(\mathcal{E}).\widetilde{\omega}^B =\sum_{A,B,C,D} M_{AB}\tau(\mathcal{E})_{AC}\widetilde{\omega}^C\wedge \tau(\mathcal{E})_{BD}\widetilde{\omega}^D\\&=\sum_{ C,D}\left (\tau(\mathcal{E})^t M\tau(\mathcal{E})\right)_{CD}\widetilde{\omega}^C\wedge  \widetilde{\omega}^D,
    \end{split}\end{equation*}
from which we get (\ref{eq:normal-form}). \hskip 112mm  $\Box$
       \vskip 5mm

{\it Proof of Theorem \ref{thm:det}}.  In order to use   results in the proof of Theorem \ref{thm:real-form}, we use notations there. Let $\mathcal{M} =-b+a\mathbf{j}$ be  a
  hyperhermitian matrix as above, and let $\omega=\sum_{A,B=0}^{2n-1} M_{ AB}\,\omega^A\wedge\omega^B$ be the $2$-element, where $
   M=(M_{ AB})=\tau(\mathcal{M})J$. Let $\mathcal{B}=   e +f\mathbf{j}$ be the unitary matrix such that (\ref{eq:-b+aj}) holds. Then
  (\ref{eq:normal-form}) holds
    with $\mathcal{E}:=\overline e+\overline{f }\mathbf{j}$ and $\widetilde {\omega}^A= \mathcal{E^*}.{\omega}^A$.

Now apply Propositions \ref{prop:Moore} (2) to determinants of both sides of (\ref{eq:-b+aj}) to get
\begin{equation}\label{eq:det-nu}
   \det( - b +a\mathbf{j})=\nu_0\cdots \nu_{n-1}.
\end{equation}   The $n$-th exterior products of both sides of   (\ref{eq:normal-form}) give  us
\begin{equation}\label{eq:wedge-nu}
   \omega \wedge\cdots\wedge\omega =2^nn! \nu_0\cdots \nu_{n-1} \widetilde{\omega}^0\wedge\widetilde{\omega}^{ n }\wedge\cdots\wedge\widetilde{\omega}^{ n-1}\wedge\widetilde{\omega}^{2n-1}.
\end{equation}
We claim that
\begin{equation}\label{eq:Omega-omega}
\widetilde{\omega}^0\wedge\widetilde{\omega}^{  n }\cdots\wedge
\widetilde{\omega}^{ n-1} \wedge
\widetilde{\omega}^{2n-1}=\Omega_{2n}  .
\end{equation}
This is because    $ \tau( \mathcal{E}^* )$ is  symplectic by Proposition \ref{prop:tau} (3), since $\mathcal{E} =\overline e+\overline{f }\mathbf{j}$ is  quaternionic unitary. So $\tau(\mathcal{E}^*).\beta_n=\beta_n$, i.e.
\begin{equation*} \sum_{ l=0}^{n-1} \widetilde{\omega}^l\wedge\widetilde{\omega}^{l+n}  =
   \sum_{ l=0}^{n-1}  {\omega}^l\wedge {\omega}^{l+n},
 \end{equation*} as in (\ref{eq:beta-inv}),
 whose $n$-th exterior products  give us (\ref{eq:Omega-omega}).

Apply (\ref{eq:Omega-omega}) to (\ref{eq:wedge-nu})
to get
\begin{equation*}\triangle_n\left(M,\cdots,M\right)=
2^n n!  \det( - b +a\mathbf{j}).
\end{equation*}
by definition of $\triangle_n$ in (\ref{eq:triangleup-n}). Namely for the $2$-form
\begin{equation*}
   \omega_{\mathcal{M}}:=\sum_{A,B} [\tau(\mathcal{M})J]_{AB} \omega^A\wedge \omega^B,
\end{equation*} associated to a
hyperhermitian $n\times n$-matrix $\mathcal{M}= - b +a\mathbf{j}$, we have
\begin{equation}\label{eq:det-1}
   \omega_{\mathcal{M} } \wedge\cdots\wedge  \omega_{\mathcal{M} } =2^nn!\det (   \mathcal{M}   )\Omega_{2n}.
\end{equation}

Applying  (\ref{eq:det-1}) to  $\mathcal{M}=\lambda_1 {\mathcal{M}_1}+\cdots+\lambda_n {\mathcal{M}_n}$ for real numbers $\lambda_1 ,\cdots,\lambda_n$, which is also hyperhermitian since each $\mathcal{M}_j$ is, we get
\begin{equation*}
   (\lambda_1\omega_{\mathcal{M}_1}+\cdots )\wedge\cdots\wedge (\lambda_1\omega_{\mathcal{M}_1}+\cdots )=2^nn!\det\left (\lambda_1 \tau({\mathcal{M}_1})J +\cdots+\lambda_n \tau( {\mathcal{M}_n})J \right)\Omega_{2n}.
\end{equation*}The coefficients of the monomial $\lambda_1 \cdots \lambda_n$ gives us
\begin{equation*}
   n!\omega_{\mathcal{M}_1} \wedge\cdots\wedge   \omega_{\mathcal{M}_n}=2^n(n!)^2\det\left (   {\mathcal{M}_1}  ,\cdots, {\mathcal{M}_n}  \right)\Omega_{2n}.
\end{equation*}
The theorem is proved. \hskip 118mm $\Box$

 \section{The   quaternionic Monge-Amp\`{e}re operator and its fundamental solution }

For a point $q=\left(\begin{array}{c}q_0\\ \vdots\\ q_{n-1}\end{array}\right) \in \mathbb{H}^n$, write
\begin{equation}\label{eq:ql}
 q_{l  }:=x_{4l}+x_{4l+1}\mathbf{i}+x_{4l+2}\mathbf{j}+x_{4l+3 }\mathbf{k},
 \end{equation}
$l=0,\ldots,n-1 $. The {\it Cauchy-Fueter
operator}  is
\begin{equation}\label{eq:Cauchy-Fueter}
   \frac {\partial}{\partial {\overline{q_l  }  }}  =\partial_{x_{4l}}+
 \mathbf{i}\partial_{x_{4l+1}}
 + \mathbf{j}\partial_{x_{4l+2}}+ \mathbf{k}\partial_{x_{4l+3  }},
\end{equation}
and its conjugate
\begin{equation}\label{eq:Cauchy-Fueter--}
\frac {\partial}{\partial { {q}_{l }  }}    : =\partial_{x_{4l }}-
 \mathbf{i} \partial_{x_{4l+1}}
 - \mathbf{j}\partial_{x_{4l+2}}- \mathbf{k}\partial_{x_{4l+3  }}.
\end{equation}  The {\it quaternionic Hessian} of a real $C^2$ function $u$ at the point $q$ is the hyperhermitian matrix
  \begin{equation}\label{eq:quaternionic-Hessian}
     \left(\frac {\partial^2 u}{\partial\overline{ q_l}\partial  {q}_k }(q)\right).
 \end{equation}

\vskip 5mm

 {\it Proof of Theorem \ref{thm:det'}}.   Note that it follows from definitions of $\nabla_{A\alpha}$'s in   (\ref{eq:nabla-jj'-new}) and $\triangle_{AB}$'s in (\ref{2.10}) that
 \begin{equation*}
   \textbf{j}\nabla_{(n+k )0'}=-\nabla_{ k1'}\textbf{j},\qquad \textbf{j}\nabla_{(n+k )1'}=\nabla_{ k  0'}\textbf{j},
\end{equation*}
 and \begin{equation*}\begin{aligned}\frac{\partial^2u}{\partial  \overline{{q_l}}\partial q_k}&=\left(\frac{\partial }{\partial x_{4l}}+\textbf{i}\frac{\partial }{\partial x_{4l+1}}+\textbf{j}\frac{\partial }{\partial x_{4l+2}}+\textbf{k}\frac{\partial }{\partial x_{4l+3}}\right)\left(\frac{\partial }{\partial x_{4k}}-\textbf{i}\frac{\partial }{\partial x_{4k+1}}-\textbf{j}\frac{\partial }{\partial x_{4k+2}}-\textbf{k}\frac{\partial }{\partial x_{4k+3}}\right)u\\&=\left(\nabla_{ l 0'}-\nabla_{l1'}\textbf{j}\right)\left(\nabla_{(n+k )1'}-\textbf{j}\nabla_{(n+k )0'}\right)u\\
&=\left(\nabla_{ l 0'}\nabla_{(n+k )1'}-\nabla_{ l 1'}\nabla_{(n+k )0'}\right)u+\left(\nabla_{ l  0'}\nabla_{ k  1'}-\nabla_{ l 1'}\nabla_{ k 0'}\right)u\,\textbf{j}\\
&=2\left(\triangle_{ l(n+k )}u+\triangle_{ l   k  }u\,\textbf{j}\right) \end{aligned}\end{equation*} for $l,k=0,\ldots,n-1$.
So the the quaternionic Hessian  can be written as
\begin{equation*}
  \left (\frac{\partial^2u}{\partial  \overline{{q_l}}\partial q_k}\right)=a+b\mathbf{j}
\end{equation*}
with $n\times n$ complex matrices
\begin{equation*}
   a:=2(\triangle_{ l(n+k )}u),\qquad b:=2(\triangle_{ l   k  }u).
\end{equation*}
Thus
\begin{equation*}\begin{split}
   \tau\left (\frac{\partial^2u}{\partial  \overline{{q_l}}\partial q_k}\right)J&=\left(
                                       \begin{array}{rr}
                                         a &- b \\
\overline{b}&\overline{a}\\
                                       \end{array} \right)J=\left(
                                       \begin{array}{rr}
                                         b &  a \\
-\overline{a}&\overline{b}\\
                                       \end{array} \right) =2\left(
                                       \begin{array}{rr}
                                        \triangle_{ l   k  }u &\triangle_{ l(n+k )}u \\
- \overline{\triangle_{ l(n+k )}u}&\overline{\triangle_{ l k }u}\\
                                       \end{array} \right) .
\end{split}\end{equation*}
We find that
\begin{equation}\label{eq:Hessian-delta}\begin{split}
   \tau\left (\frac{\partial^2u}{\partial  \overline{{q_l}}\partial q_k}\right)J =2\left(\triangle_{ AB  }u\right),
\end{split}\end{equation}
because
\begin{equation*}
   \qquad\overline{ {\triangle}_{l(n+k )} {u}}=\overline{\nabla_{ l 0'}}\overline{\nabla_{(n+k )1'}u}-\overline{\nabla_{ l 1'}}\overline{\nabla_{(n+k )0'}u}= {\nabla_{(n+l )1'}}{\nabla_{ k 0'}u}- {\nabla_{(n+l )0'}}{\nabla_{ k 1'}u}= -{\triangle}_{(n+l ) k }  {u}
\end{equation*} by definition of $\nabla_{A\alpha}$'s in (\ref{eq:nabla-jj'-new}). Similarly,
\begin{equation*}
   \overline{ {\triangle}_{lk} {u}}= {\triangle}_{(n+l )(n+k )}{u}.
\end{equation*}
The result follows from (\ref{eq:Hessian-delta}) by applying Theorem \ref{thm:det} to $\mathcal{M}_j=(\frac{\partial^2u_j}{\partial \overline{{q_l}}\partial q_k} ) $. \hskip 33mm$\Box$

\vskip 5mm

The
first-order differential operators $\nabla_{A\alpha}$'s in (\ref{eq:nabla-jj'-new}) are again the relabelling of the following differential operators   we used in \cite{wan-wang}:
\begin{equation}\label{eq:nabla-jj'}\left(
                             \begin{array}{cc}
                               \nabla_{00' } & \nabla_{01' } \\
                               \nabla_{10' } & \nabla_{11' } \\
                               \vdots&\vdots\\
                               \nabla_{(2l)0' } & \nabla_{(2l)1' } \\
                               \nabla_{(2l+1)0' } & \nabla_{(2l+1)1' } \\
                               \vdots&\vdots
                             \end{array}
                           \right):=\left(
                                      \begin{array}{cc}
                                      \partial_{x_{0}}+\textbf{i}\partial_{x_{1}} & -\partial_{x_{2}}-\textbf{i}\partial_{x_{3}} \\
                                        \partial_{x_{2}}-\textbf{i}\partial_{x_{3}} & \partial_{x_{0}}-\textbf{i}\partial_{x_{1}} \\
                                         \vdots&\vdots\\
                                         \partial_{x_{4l}}+\textbf{i}\partial_{x_{4l+1}} & -\partial_{x_{4l+2}}-\textbf{i}\partial_{x_{4l+3}} \\
                                        \partial_{x_{4l+2}}-\textbf{i}\partial_{x_{4l+3}} & \partial_{x_{4l}}-\textbf{i}\partial_{x_{4l+1}} \\
                                        \vdots&\vdots
                                      \end{array}
                                    \right).
\end{equation}

Let $\Omega$ be a domain in $  \mathbb{H}^n$. Recall that $d_0,d_1:C_0^\infty(\Omega,\wedge^{p}\mathbb{C}^{2n})\rightarrow C_0^\infty(\Omega,\wedge^{p+1}\mathbb{C}^{2n})$ are given by
\begin{equation}\label{eq:d}\begin{aligned}&
d_0F:=
\sum_{ I}  \sum_{A =0}^{2n-1}   \nabla_{A0' }f_{I}~\omega^A\wedge\omega^I,\\
&d_1F:=\sum_{ I}\sum_{A =0}^{2n-1}\nabla_{A1' }f_{I}~\omega^A\wedge\omega^I,
\end{aligned}\end{equation}
for $F=\sum_{I}f_{I}\omega^I\in C_0^\infty(\Omega,\wedge^{p}\mathbb{C}^{2n})$,  where $\omega^I:=\omega^{i_1}\wedge\ldots\wedge\omega^{i_{p}}$ for the multi-index
$I=(i_1,\ldots,i_{p})$.
We say $F$
is \emph{closed} if
\begin{equation*} d_0F=d_1F=0.
\end{equation*} Then we have
\begin{equation*}
  \triangle u=d_0d_1u .
\end{equation*}

\begin{prop}\label{p1.1} {\rm (Proposition 2.2 in \cite{wan-wang})} $($1$)$ $d_0d_1=-d_1d_0$.\\
$($2$)$ $d_0^2=d_1^2=0$.\\
$($3$)$ For $F\in C_0^\infty(\Omega,\wedge^{p}\mathbb{C}^{2n})$, $G\in C_0^\infty(\Omega,\wedge^{q}\mathbb{C}^{2n})$, we have\begin{equation*}d_\alpha(F\wedge G)=d_\alpha F\wedge G+(-1)^{p}F\wedge d_\alpha G,\qquad \alpha=0 ,1 .\end{equation*}
\end{prop}

$z^{A\alpha}$'s given by
\begin{equation}\label{2.3}\left(
                             \begin{array}{cc}
                               z^{00' } & z^{01' } \\
                               \vdots&\vdots\\
                               z^{( l)0' } & z^{( l)1' } \\ \vdots&\vdots
                           \\z^{n0' } & z^{n1' } \\
                               \vdots&\vdots\\
                               z^{( n+l)0' } & z^{( n+l)1' } \\ \vdots&\vdots                                                         \end{array}
                           \right):=\left(
                                      \begin{array}{cc}
                                      x_{0}-\textbf{i}x_{1} & -x_{2}+\textbf{i}x_{3} \\
                                                                               \vdots&\vdots\\
                                         x_{4l}-\textbf{i}x_{4l+1} & -x_{4l+2}+\textbf{i}x_{4l+3} \\
                                     \vdots&\vdots\\ x_{2}+\textbf{i}x_{3} & x_{0}+\textbf{i}x_{1} \\
                                        \vdots&\vdots\\
                                        x_{4l+2}+\textbf{i}x_{4l+3} &  x_{4l }+\textbf{i}x_{4l+1} \\
                                          \vdots&\vdots
                                      \end{array}
                                    \right)
\end{equation}
can be viewed as independent variables formally and $\nabla_{A\alpha}$'s are derivatives with respect to these variables by the following lemma.
\begin{lem}\label{lem:nabla-z} {\rm (Lemma 3.1 and Corollary 3.1 (1) in \cite{wan-wang})} (1) $\nabla_{A\alpha}z^{B\beta}=2\delta_A^{B}\delta_{\alpha}^{\beta}$.

(2) $\nabla_{A\alpha}(\|q\|^2)=2\overline{z^{A\alpha}}.$
\end{lem}

The function $-\frac{1}{\|q\|^2}$ is a PSH function (Proposition 4.1 in \cite{wan-wang}), and  can be shown to be the fundamental solution of the quaternionic Monge-Amp\`{e}re operator in   the following simple way by using $d_0$ and $ d_1$.

\begin{prop}\label{p4.2}  {\rm (Proposition 4.1 in \cite{wan-wang})}   $\left(\triangle\left(-\frac{1}{\|q\|^2}\right)\right)^n=\frac{8^nn!\pi^{2n}}{(2n)!}\delta_0\Omega_{2n}$.\end{prop}
\begin{proof} Note that
\begin{equation*}d_\alpha\left(-\frac{1}{\|q\|^2+\varepsilon}\right)
= \frac{d_\alpha \|q\|^2}{(\|q\|^2+\varepsilon)^2},
\end{equation*}
and so
\begin{equation*}\bigtriangleup \left(-\frac{1}{\|q\|^2+\varepsilon}\right)= d_0d_1\left(-\frac{1}{\|q\|^2+\varepsilon}\right)
= -\frac{2 d_0 \|q\|^2\wedge d_1\|q\|^2}{(\|q\|^2+\varepsilon)^3}+\frac{d_0d_1 \|q\|^2}{(\|q\|^2+\varepsilon)^2},
\end{equation*}where
\begin{equation}\label{eq:dbeta}
  d_\alpha \|q\|^2=2\sum_{A=0}^{2n-1}  \overline{z^{A\alpha}}\omega^A,
\end{equation} by using Lemma \ref{lem:nabla-z} (2).
In particular, we have
\begin{equation*}
  d_1 \|q\|^2= 2\sum_{l=0}^{ n-1} \left (- {z^{(n+l)0' }}\omega^l+{z^{l0' }}\omega^{n+l}\right),
\end{equation*}
and so
\begin{equation}\label{eq:beta-d}
   d_0d_1 \|q\|^2=8\beta_n,
\end{equation}
by using Lemma \ref{lem:nabla-z} (1). It follows that
\begin{equation}\label{eq:delta-fund}\left(\bigtriangleup \left(-\frac{1}{\|q\|^2+\varepsilon}\right)\right)^n
= -\frac{2\cdot 8^{n-1}n}{(\|q\|^2+\varepsilon)^{2n+1}}d_0 \|q\|^2\wedge d_1\|q\|^2\wedge \beta_n^{n-1}+\frac{8^n}{(\|q\|^2+\varepsilon)^{2n}}\beta_n^{n }
\end{equation}
by $\omega\wedge\omega=0$ for any $1$-form $\omega$ and (\ref{eq:beta-d}).   Note that $\beta_n^n=n!\Omega_{2n} $ and
\begin{equation*}
 {{z^{l0'}}{z^{(n+l)1'}}}- {{z^{(n+l)0'}}{z^{l1'}}}= |q_l |^2,
\end{equation*}
from which we see that
\begin{equation}\label{eq:d-d}
   d_0 \|q\|^2\wedge d_1\|q\|^2=4\sum_{l=0}^{n-1}  |q_l |^2\omega^l\wedge\omega^{n+l}+\sum_{|j-k|\neq n} a_{jk}\omega^j\wedge\omega^{k}.
\end{equation}
Substitute  (\ref{eq:d-d}) into
 (\ref{eq:delta-fund}) to get
 \begin{equation}\label{eq:delta0} \left(\bigtriangleup \left(-\frac{1}{\|q\|^2+\varepsilon}\right)\right)^n =\frac{8^n n!\varepsilon}{(\|q\|^2+\varepsilon)^{2n+1}} \Omega_{2n}.\end{equation}

Let $\varepsilon=0$, we get \begin{equation*}
\left(\triangle\left(-\frac{1}{\|q\|^2}\right)\right)^n=0,\end{equation*}for $q\neq0$. Then we can check $\left(\triangle\left(-\frac{1}{\|q\|^2}\right)\right)^n=\frac{8^nn!\pi^{2n}}{(2n)!}\delta_0\Omega_{2n}$
 by integrating (\ref{eq:delta0})  as in the proof  of Proposition 4.1 in \cite{wan-wang}.
\end{proof}

\section{The invariance of $d_\alpha$ and $\triangle$ under quaternionic linear transformations}

The operators $d_0$ and $d_1$ in (\ref{eq:d}) depend on the choice of the coordinates $x_j$'s and the basis $\{\omega^A \}$. Let us give their transformation formula.

For a  quaternionic vector $q=\left(\begin{array}{l}q_0\\ \vdots\\ q_{n-1}\end{array}\right) \in \mathbb{H}^n$, write
\begin{equation}\label{eq:x-beta}
     q = x^{(0)}+ \mathbf{i}x^{(1)}+ \mathbf{j}x^{(2)}+ \mathbf{k}x^{(3)} ,\qquad {\rm
 with}\qquad
       x^{(\beta)}=\left(\begin{array}{c}  \vdots \\x_{4l+\beta}\\\vdots \end{array}\right) \in   \mathbb{R}^n,
\end{equation}
$\beta=0,1,2,3 $,
 where $x_0,\ldots,x_{4n-1}$ are real coordinates of $q$ given by (\ref{eq:ql}).
A quaternionic   matrix $\mathcal{U}\in M_{\mathbb{H}}(n,n)$  defines a right linear transformation   $\mathbb{H}^{4n}\rightarrow \mathbb{H}^{4n}$ given by $\widetilde{q}= \mathcal{U}q$.
Write
\begin{equation}\label{eq::X-AB}
    \mathcal{U}=U_0+ \mathbf{i}U_1+ \mathbf{j}U_2+ \mathbf{k}U_3,
\end{equation}
for some real $(n\times n)$-matrices $ U_0,U_1, U_2,U_3$.
If we denote
\begin{equation}\label{eq:X-R}
 \mathcal{{U}}^{\mathbb{R}}:=     \left(
\begin{array}{rrrr}U_0&-  U_1&-  U_2&- U_3 \\    U_1&U_0&- U_3&U_2 \\  U_2& U_3& U_0&- U_1 \\ U_3&-U_2&   U_1&U_0 \end{array}\right),\qquad  {q}^{\mathbb{R}}:=\left(\begin{array}{c} x^{(0)}\\x^{(1)}\\ x^{(2)} \\x^{(3)} \end{array}\right)\in\mathbb{R}^{4n} ,
\end{equation}
it is direct to check that
\begin{equation}\label{eq:widetildeX-q}
     {(\mathcal{U} q)}^{\mathbb{R}}=  {\mathcal{U}}^{\mathbb{R}} {q}^{\mathbb{R}},
\end{equation}
(cf. (2.9)-(2.10) in \cite{wang1}).
\begin{prop}\label{prop:U*} Let $\widetilde{\nabla}_{A\alpha}$ be the operators defined by (\ref{eq:nabla-jj'-new}) in terms of the real coordinates $\widetilde{x}$ of $\widetilde{q}=\mathcal{U}q$, i.e. $ \widetilde{q}_l=\widetilde{x}_{4l}+\textbf{i}\widetilde{x}_{4l+1}+\textbf{j}\widetilde{x}_{4l+2}+\textbf{k}\widetilde{x}_{4l+3} $.
Under the transformation $\mathcal{U}$, write $u(q):=\widetilde{u}(\mathcal{U}q)$. Then we have
\begin{equation}\label{eq:U*}
    \left.\nabla_{A\alpha} u\right|_q=\sum_{B=0}^{2n-1}\left.\overline{\tau({\mathcal{U}})}_{B A}\widetilde{\nabla}_{B\alpha} \widetilde{u}\right|_{\mathcal{U}q},
\end{equation}for $\alpha=0',1'$, $A=0,\ldots, 2n-1$.
\end{prop}  \begin{proof} In this proof, let $\hat x_a$ be the $a$-th entry of the vector $ {q}^{\mathbb{R}}$ in (\ref{eq:X-R}). It is just relabeling of coordinates $(x_0,\ldots, x_{4n-1})$. Similarly, $\hat {\widetilde{x}}_a$ be the $a$-th entry of the corresponding vector $\widetilde{q}^{\mathbb{R}}$   for $\widetilde{q}$.
By definition, we have
\begin{equation*}\partial_{\hat x_a}\left[\widetilde{u}  \left({\mathcal{U}}^{\mathbb{R}} {q}^{\mathbb{R}}\right)\right]
  = \sum_{b=0}^{4n-1} \left(\mathcal{U}^{\mathbb{R}}\right)_{ba} (\partial_{\hat{ \widetilde{x}}_b}\widetilde{u})(\mathcal{U}q).
\end{equation*}Note that $x^{(0)}_j=\hat x_j$, $x^{(1)}_j=\hat x_{n+j}$, for  $j=0,\ldots,n-1$. Therefore, we have
 \begin{equation}\label{eq:nabla-U}\begin{split}
    \nabla_{j0'}u=& \partial_{x^{(0)}_j}u+\mathbf{i} \partial_{x^{(1)}_j}u=\sum_{b=0}^{4n-1} \left \{\left(\mathcal{U}^{\mathbb{R}}\right)_{bj} \partial_{\hat{ \widetilde{x}}_b}+\mathbf{i}\left(\mathcal{U}^{\mathbb{R}}\right)_{b(n+j)} \partial_{\hat{ \widetilde{x}}_b}\right\}\widetilde{u}\\
    =&\sum_{l=0}^{ n-1}  \left \{(U_0)_{lj}\partial_{\widetilde{x}_{4l}}+(U_1)_{lj}\partial_{\widetilde{x}_{4l+1}}+(U_2)_{lj}\partial_{\widetilde{x}_{4l+2}}
    +(U_3)_{lj}\partial_{\widetilde{x}_{4l+3}}\right\}\widetilde{u}\\
    &+\mathbf{i}\sum_{l=0}^{ n-1}  \left\{-(U_1)_{lj}\partial_{\widetilde{x}_{4l}}+(U_0)_{lj}\partial_{\widetilde{x}_{4l+1}}
    +(U_3)_{lj}\partial_{\widetilde{x}_{4l+2}}-(U_2)_{lj}\partial_{\widetilde{x}_{4l+3}}\right\}\widetilde{u}\\
    =&\sum_{l=0}^{ n-1} \left\{ (U_0-\mathbf{i} U_1)_{lj}(\partial_{\widetilde{x}_{4l}}+\mathbf{i}\partial_{\widetilde{x}_{4l+1}})+( U_2+\mathbf{i} U_3)_{lj}(\partial_{\widetilde{x}_{4l+2}}-\mathbf{i}\partial_{\widetilde{x}_{4l+3}})\right\}\widetilde{u}\\
    =&\sum_{l=0}^{ n-1} \left\{ (U_0-\mathbf{i} U_1)_{lj}\widetilde{\nabla}_{l0'}+( U_2+\mathbf{i} U_3)_{lj}\widetilde{\nabla}_{(n+l)0'}\right\}\widetilde{u}
\end{split} \end{equation}
by the first two columns of $\mathcal{U}^{\mathbb{R}}$ in (\ref{eq:X-R}) and $\nabla_{j\alpha}$'s in (\ref{eq:nabla-jj'-new}).
Similarly, we have
 \begin{equation}\label{eq:nabla-U'}\begin{split}
    \nabla_{(n+j)0'}u=& \partial_{x^{(2)}_j}u-\mathbf{i} \partial_{x^{(3)}_j}u\\
    =&\sum_{l=0}^{ n-1}  \left \{(- U_2+\mathbf{ {i}} U_3)_{lj}(\partial_{\widetilde{x}_{4l}}+\mathbf{i}\partial_{\widetilde{x}_{4l+1}})+(U_0+\mathbf{i} U_1)_{lj}(\partial_{\widetilde{x}_{4l+2}}-\mathbf{i}\partial_{\widetilde{x}_{4l+3}})\right\}\widetilde{u}.
\end{split} \end{equation}
Note that
\begin{equation*}\label{eq:tau-U}\overline{\tau(\mathcal{U})} :=\left(
                                       \begin{array}{rr}
                                         U_0 - \textbf{i} U_1 & -U_2+\textbf{i} U_3  \\
                                    U_2+\textbf{i} U_3     &  U_0+\textbf{i} U_1  \\
                                       \end{array}
                                     \right).\end{equation*}
It follows from (\ref{eq:nabla-U})-(\ref{eq:nabla-U'}) that  (\ref{eq:U*}) holds for $\alpha=0'$. Then (\ref{eq:U*})  for $\alpha=1'$ follows by taking conjugate on both sides of (\ref{eq:nabla-U})-(\ref{eq:nabla-U'}) and noting that $\overline{\nabla_{j0'}}=\nabla_{(n+j)1'} $, $\overline{\nabla_{(n+j)0'}}=-\nabla_{j1'}$ and so for $\widetilde{\nabla}$.
\end{proof}

For a basis
$\{\omega^0,\omega^1,\ldots,\omega^{2n-1}\}$ of $\mathbb{C}^{2n}$, $\{\widetilde{\omega}^0,\widetilde{\omega}^1,\ldots,\widetilde{\omega}^{2n-1}\}$ given by
\begin{equation*}
   \widetilde{\omega}^A:= \overline{\tau(\mathcal{U} )}. \omega^A =\sum_{B=0}^{2n-1}\overline{\tau(\mathcal{U} )}_{AB} \omega^B
\end{equation*}is also a basis for $\mathcal{U}\in \text{GL}_{\mathbb{H}}(n)$.

\begin{prop}\label{prop:d-trans} {\rm (Proposition 2.4 in \cite{wan-wang})}  Let $u$   be a real $ C^1$ functions. Then $d_\alpha u$ is
invariant under a quaternionic linear transformations $\mathcal{U}$ on $\mathbb{H}^n$, i.e.,
\begin{equation*}
  \left.  d_\alpha u\right|_q= \widetilde{d_\alpha} \left. \widetilde{u}\right |_{\mathcal{U}q},
  \end{equation*}
for $\alpha=0',1'$, where $u(q):=\widetilde{u}(\mathcal{U}q)$,  $\widetilde{d_\alpha} \widetilde{u}=\sum_A\widetilde{\nabla}_{A\alpha}\widetilde{u}~\widetilde{\omega}^A$.
\end{prop}
\begin{proof}By Proposition \ref{prop:U*}, we get
 \begin{equation*}
   d_\alpha u=\sum_A  \nabla_{A\alpha} {u}\,  \omega^A= \sum_{A,B} \overline{\tau(\mathcal{U})}_{B A}\widetilde{\nabla}_{B\alpha} \widetilde{u} \,  \omega^A=\sum_B  \widetilde{\nabla}_{B\alpha} \widetilde{u}\,  \widetilde{\omega }^B=\widetilde{d}_\alpha \widetilde{u}.
\end{equation*}
 The transformation formula follows.
\end{proof}

\begin{cor}\label{prop:delta} {\rm (Corollary 2.2 in \cite{wan-wang})}  $\triangle u$ and $(\triangle u)^n$ are also invariant under quaternionic linear transformations on $\mathbb{H}^n$, i.e., $\triangle u|_{ q}=\widetilde{\triangle} \widetilde{u}|_{\mathcal{U}q}$.\end{cor}
\begin{proof} We have \begin{equation*}\begin{aligned}\triangle u(q)&=d_0d_1u(q)=d_0 \left[(\widetilde{d_1}\widetilde{u})(\mathcal{U}q)\right]=\sum_{A,B}\nabla_{A0'}\left[(\widetilde{\nabla}_{B1'}\widetilde{u})(\mathcal{U}q)\right]~\omega^A
\wedge\widetilde{\omega}^B\\
&=\sum_{A,B,C}\overline{\tau(\mathcal{U})}_{C A}\left(\widetilde{\nabla}_{C0'}\widetilde{\nabla}_{B1'}\widetilde{u}\right)(\mathcal{U}q)~\omega^A
\wedge\widetilde{\omega}^B
=\sum_{ B,C} \left(\widetilde{\nabla}_{C0'}\widetilde{\nabla}_{B1'}\widetilde{u}\right)(\mathcal{U}q)~\widetilde{\omega}^C\wedge\widetilde{\omega}^B\\&
=(\widetilde{\triangle} \widetilde{u})(\mathcal{U}q) \end{aligned}\end{equation*}by applying Proposition \ref{prop:U*}-\ref{prop:d-trans}.\end{proof}

Hence $\triangle_n u$ is also invariant under quaternionic unitary transformation $\mathcal{U}$ by $\mathcal{U}.\Omega_{2n}=\Omega_{2n}$, i.e. \begin{equation*}
  \triangle_n(u_1,\ldots,u_n)(q)=\widetilde{\triangle}_n(\widetilde{u}_1,\ldots,\widetilde{u}_n)(\mathcal{U}q).
\end{equation*}


\begin{thebibliography}{20}

\bibitem{alesker1}{\sc S. Alesker}, {Non-commmutative linear algebra and plurisubharmonic functions of quaternionic variables}, {\it
Bull. Sci. Math.} \textbf{127}(1) (2003), 1-35.

\bibitem{alesker4}{\sc S. Alesker}, {Quaternionic Monge-Amp\`{e}re equations}, {\it J. Geom. Anal.} \textbf{13} (2003), no. 2, 205-238.




\bibitem{alesker2}{\sc S. Alesker}, {Pluripotential theory on quaternionic manifolds,} {\it J. Geom. Phys.} \textbf{62} (2012), no. 5, 1189-1206.



  \bibitem{AS} {\sc
    S. Alesker and E. Shelukhin},  On a uniform estimate for the quaternionic Calabi problem, {\it Israel J. Math.} \textbf{ 197},
(2013), 309-327.
\bibitem{AS2} {\sc
    S. Alesker and E. Shelukhin},
A uniform estimate for general quaternionic Calabi problem (with appendix by Daniel Barlet), {\it Adv. in Math.} {\bf   316} (2017), 1-52.

\bibitem{alesker6} {\sc S. Alesker and M. Verbitsky}, {Plurisubharmonic functions on hypercomplex manifolds and HKT-geometry}, {\it J. Geom. Anal.} {\bf16} (2006), 375-399.

\bibitem{baston}{\sc  R. J. Baston}, {Quaternionic complexes}, {\it J. Geom. Phys.} \textbf{8} (1992), no. 1-4, 29-52.


\bibitem{bed}{\sc E. Bedford and B.A. Taylor}, {A new capacity for plurisubharmonic functions}, {\it Acta Math.} \textbf{149} (1982), no. 1-2, 1-40.
\bibitem{bS}{\sc J. Bure\v s  and V. Sou\v cek},
  Complexes of invariant differential operators in
several quaternionic variables, {\it  Complex Var. Elliptic Equ.\/} {\bf
51} (2006), no. 5-6, 463-487.

\bibitem{harvey1}{\sc    R. Harvey and H.  Lawson}, {An introduction to potential theory in calibrated geometry}, {\it Amer. J. Math.} \textbf{131} (2009), 893-944.
\bibitem{harvey2}{\sc    R. Harvey and H.   Lawson}, {Duality of positive currents and plurisubharmonic functions in calibrated geometry}, {\it Amer. J. Math.} \textbf{131} (2009), 1211-1239.


\bibitem{klimek}{\sc M. Klimek},  {\it  Pluripotential Theory},
 Clarendon Press (1991).

  \bibitem{wan4}{\sc	D.  Wan},
  The continuity and range of the quaternionic Monge-Amp\`ere operator on quaternionic space,  {\it Math. Z.} {\bf 285}  (2017), 461-478.


\bibitem{wan5}{\sc	D.    Wan and Q.   Kang}, Potential theory for quaternionic plurisubharmonic functions,  {\it  Michigan Math. J.}  {\bf   66}  (2017),  no. 1, 3-20.


\bibitem{wan-wang} {\sc D.    Wan and W.~Wang},  On quaternionic Monge-Amp\`{e}re operator, closed positive currents and Lelong-Jensen type formula on the quaternionic space, {\it
Bull. Sci. Math.} {\bf  141} (2017), 267-311.

\bibitem{wan3} {\sc	D.    Wan and W. Zhang}, Quasicontinuity and maximality of quaternionic plurisubharmonic functions, {\it J. Math. Anal. Appl.} {\bf 424} (2015), 86-103.



 \bibitem{Wang} {\sc W. Wang},
  {The $k$-Cauchy-Fueter complex, Penrose transformation and Hartogs' phenomenon for quaternionic $k$-regular functions},
{\it J. Geom. Phys.} {\bf 60} (2010), 513-530.


\bibitem{wang1}{\sc W. Wang}, {On the optimal control method in quaternionic analysis,} {\it Bull. Sci. Math.} \textbf{135} (2011), no. 8, 988-1010.

\bibitem{wang-mfd}{\sc W. Wang},
 On quaternionic complexes over unimodular quaternionic manifolds,
{\it Diff. Geom.     Appl.}  58 (2018),  227-253.

\bibitem{Zhu}{\sc J. Zhu},
Dirichlet problem of quaternionic Monge-Amp\`re equations, {\it Israel J. Math.} {\bf  214} (2016),  597-619.





\end{thebibliography}
 \end{document}